\documentclass{amsart}

\usepackage{url,amssymb,enumerate,colonequals}

\usepackage{mathrsfs} 
\usepackage[section]{placeins}
\usepackage{MnSymbol}
\usepackage{extarrows}
\usepackage{lscape}
\usepackage[all,cmtip]{xy}

\usepackage[OT2,T1]{fontenc}
\usepackage{color}

\usepackage[
        colorlinks, citecolor=darkgreen,
        backref,
        pdfauthor={Nuno Freitas}, 
]{hyperref}

\usepackage{comment}
\usepackage{multirow}

\newcommand{\F}{\mathbb{F}}

\newcommand{\Q}{\mathbb{Q}}

\newcommand{\Z}{\mathbb{Z}}

\newcommand{\fq}{\mathfrak{q}}

\newcommand{\fP}{\mathfrak{P}}

\newcommand{\calO}{\mathcal{O}}

\newcommand{\fp}{\mathfrak{p}}

\DeclareMathOperator{\Cl}{Cl}

\DeclareMathOperator{\Gal}{Gal}

\DeclareMathOperator{\Norm}{Norm}

\DeclareMathOperator{\ord}{ord}

\DeclareMathOperator{\Res}{Res}

\numberwithin{equation}{section}

\newtheorem{theorem}{Theorem}
\newtheorem*{thm}{Theorem}
\newtheorem{lemma}{Lemma}

\newtheorem{corollary}{Corollary}

\theoremstyle{definition}

\newtheorem{example}[equation]{Example}

\theoremstyle{remark}
\newtheorem{remark}[equation]{Remark}

\definecolor{darkgreen}{rgb}{0,0.5,0}

\DeclareRobustCommand{\SkipTocEntry}[5]{}

\begin{document}

\title[]{
Chevalley's class number formula, unit equations\\ and 
the asymptotic Fermat's Last Theorem}

\author{Nuno Freitas}

\address{Departament de Matem\`atiques i Inform\`atica,
Universitat de Barcelona (UB),
Gran Via de les Corts Catalanes 585,
08007 Barcelona, Spain}

\email{nunobfreitas@gmail.com}

\author{Alain Kraus}
\address{Sorbonne Universit\'e,
Institut de Math\'ematiques de Jussieu - Paris Rive Gauche,
UMR 7586 CNRS - Paris Diderot,
4 Place Jussieu, 75005 Paris, 
France}
\email{alain.kraus@imj-prg.fr}

\author{Samir Siksek}

\address{Mathematics Institute\\
	University of Warwick\\
	CV4 7AL \\
	United Kingdom}

\email{s.siksek@warwick.ac.uk}

\date{\today}
\thanks{Freitas is supported by a Ram\'on y Cajal fellowship with reference RYC-2017-22262.
Siksek is supported by the
EPSRC grant \emph{Moduli of Elliptic curves and Classical Diophantine Problems}
(EP/S031537/1).}

\begin{abstract}
 Let~$F$ be a number field and $\calO_F$ its ring of integers. 
We use Chevalley's ambiguous class number formula to give a criterion
for the non-existence of solutions to the unit equation
$\lambda + \mu = 1$, $\lambda, \mu \in \calO_F^\times$. 
This is then used to strengthen a criterion for the asymptotic Fermat's Last
Theorem due to Freitas and Siksek.
\end{abstract}

\maketitle

\section{Introduction}

Let~$F$ be a number field
and consider the unit equation
\begin{equation}\label{E:unit}
 \lambda + \mu = 1, \quad \lambda, \; \mu \in \calO_F^\times.
\end{equation}
This has finitely many solutions by
a famous theorem of
Siegel \cite{Siegel}. Unit equations are the subject 
of extensive research, both from a theoretical
and a computational point-of-view. A beautiful 
result due to Evertse \cite{Evertse}
asserts that if $F$ has signature $(r_1,r_2)$ then \eqref{E:unit}
has at most $3 \times 7^{3 r_1+4r_2}$ solutions. For an extensive
survey, see \cite{EvertseGyory}.

We are in fact interested in sufficient criteria on $F$
that guarantee \eqref{E:unit} has no solutions.
For example, a recent spectacular result 
of Triantafillou \cite{Triantafillou}
asserts if $3$ totally splits in $F$
and $3 \nmid [F:\Q]$ then \eqref{E:unit} has no solutions.
The authors of the current paper have shown \cite{FKS2} 
that if $F$ is a Galois $p$-extension, 
where $p \ge 5$ is prime that totally ramifies
in $F$, then \eqref{E:unit} has no solutions. 

In this paper, we use 
Chevalley's ambiguous class number formula to
give a criterion (Theorem~\ref{T:th1})
for the non-existence of solutions to \eqref{E:unit}.
This is then used to strengthen a criterion for 
the asymptotic Fermat's Last Theorem established
in \cite{FS1}.

\begin{theorem}  \label{T:th1}  
Let $F$ be a number field with class number~$h_F$. Let $p\geq 3$ be a prime number totally ramified in $F$ and set $r=\gcd(2 h_F,p-1)$. 
Suppose the following conditions hold:
 \begin{enumerate}
 \item[(i)] $\gcd\left([F:\Q],\frac{p-1}{2}\right)=1$ and
\smallskip
  \item[(ii)] 
  $\Res\left(X^r-1,(X-1)^r-1\right)\not\equiv 0 \pmod p$, where $\Res$ denotes the resultant.
\end{enumerate}
  Then the unit equation~\eqref{E:unit} has no solutions.
\end{theorem}
 
The following two corollaries give an illustration of the power of Theorem~\ref{T:th1}.
\begin{corollary} \label{C:cor1}
 Suppose  there is a prime $p\geq  5$  totally ramified in $F$ satisfying:
 \begin{enumerate}
  \item[(i)]  $\gcd\left([F:\Q],\frac{p-1}{2}\right)=1$ and
 \item[(ii)]  $\gcd(2h_F,p-1)=2$. 
  \end{enumerate}
  Then the unit equation~\eqref{E:unit} has no solutions.
 \end{corollary}
\begin{proof}
In this case $r=2$. We note that $\Res\left(X^2-1,(X-1)^2-1\right)=3$.
\end{proof}

\begin{corollary} Suppose $5$ is totally ramified in $F$ and $h_F$, $[F:\Q]$ are both odd.
  Then the unit equation~\eqref{E:unit} has no solutions.
\end{corollary} 

\begin{example} Let $F=\Q(\alpha)$ where
$\alpha^7-\alpha^6-45\alpha^5-65\alpha^4+85\alpha^3+129\alpha^2+9\alpha-11=0$.
This is a totally real field with discriminant $D_F=2^6 \cdot 7^4 \cdot 53^6$
where $p=53$ is totally ramified. We have $h_F=2$, so $r=4$. In this case, we
have $\Res(X^4 - 1,(X-1)^4 -1) = -375 = 3 \cdot 5^3$, hence the unit equation
has no solution by Theorem~\ref{T:th1}. 
\end{example}

\subsection*{The asymptotic FLT}

Let $F$ be a totally real number field. The \textbf{asymptotic Fermat's Last Theorem over $F$} is the statement that there exists a 
constant~$B_F$, depending only on~$F$, such that, for all primes~$\ell > B_F$, the only solutions to the equation
$x^\ell + y^\ell + z^\ell = 0$, with $x$, $y$, $z \in F$ are the trivial ones
satisfying $xyz =0$.  In \cite{FS1}, a criterion for asymptotic FLT over $F$
is established, which is expressed in terms of solutions
to a certain $S$-unit equation. In \cite{FKS}, the theory
of $2$-extensions is used to verify the $S$-unit criterion and therefore
asymptotic FLT for many fields $F$ where $2$ is totally ramified.
In the present paper we deal with case where $2$ is inert in $F$. 

We denote the class number of $F$ by $h_F$
and the narrow class number of~$F$ by~$h_F^+$.
Recall $h_F \mid h_F^+$ and the ratio $h_F^+/h_F$ is a power of $2$.
\begin{theorem} \label{T:th2} 
Let~$F$ be a totally real number field of odd degree. Suppose that $2$ is inert in~$F$ and 
that  $h_F^+$ is odd.
Assume further that the unit equation~\eqref{E:unit} has no solutions. 
Then the asymptotic FLT over~$F$ holds.
\end{theorem}
For a version of this theorem that also applies
to even degree fields, see Theorem~\ref{T:th2general}
below.

\begin{corollary} \label{C:cor3}
Let $F$ be totally real and $p \ge 5$ a prime totally ramified
in $F$. Suppose the following hold.
 \begin{enumerate}
  \item[(i)]  $\gcd\left(h_F \cdot [F:\Q],\frac{p-1}{2}\right)=1$.
 \item[(ii)]  $h_F^+$ and $[F:\Q]$ are both odd.
 \item[(iii)] $2$ is inert in $F$.
  \end{enumerate}
 Then the asymptotic Fermat's Last Theorem holds over $F$. 
\end{corollary}
\begin{proof} 
The result follows from Corollary~\ref{C:cor1} and Theorem~\ref{T:th2}.
\end{proof}
\begin{example} Let $F=\Q(\alpha)$ 
where $\alpha^5-110\alpha^3-605\alpha^2-990\alpha-451=0$. The field~$F$ is totally real and cyclic 
 over~$\Q$. We have $h_F=h_F^+=5$ and $p=5$ is totally ramified in~$F$.
Moreover,  $2$ is inert in $F$, hence the asymptotic FLT holds over $F$. 
\end{example}

\section{Proof of Theorem~\ref{T:th1}}
 
We will write~$\Q(\zeta_p)$ be  the $p$-th cyclotomic field and $\Q(\zeta_p)^+$
its maximal totally real subfield.  We will make use of the \emph{ambiguous
class number formula} of Chevalley (see~\cite[Chapter~13, Lemma~4.1]{Lang}):
\begin{thm}[Chevalley] 
Let $K/F$ be a cyclic extension of number fields. Then
\begin{equation}
\label{E:chevalley} 
\# \Cl(K)^G \; = \; 
\frac{h_F \cdot e(K/F)}{[K:F] \cdot [\calO_F^\times : \calO_F^\times \cap \Norm(K^\times)]} \; ,
\qquad\qquad e(K/F) = \prod_{\upsilon} e_\upsilon
\end{equation}
where $e_\upsilon$ is the ramification degree at $\upsilon$
and the product is taken over all places $\upsilon$ of $K$. 
\end{thm}

We will also need the following auxiliary lemma.

\begin{lemma} \label{L:units2}
Let $K/F$ be a Galois extension of degree~$(p-1)/2$. 
Let~$p$ be a rational prime, totally ramified in~$K$, and let $\fp$ be the unique prime of $F$ above $p$.
Then all $\varepsilon \in \calO_F^\times \cap \Norm(K^\times)$ satisfy $\varepsilon \equiv \pm 1 \pmod{\fp}$.
\end{lemma}
\begin{proof}
Write $\fP$ for the unique prime of $K$ above $\fp$.
Suppose $\varepsilon \in \calO_F^\times \cap \Norm(K^\times)$.
Then $\varepsilon=\Norm(\eta)$ where $\eta \in K^\times$. Moreover,
since $\ord_\fp(\varepsilon)=0$ and there is only one prime
$\fP$ above $\fp$, we have $\ord_\fP(\eta)=0$. Now
as $\fP/\fp$ is totally ramified, the inertia
group of $\fP/\fp$ is equal to $G=\Gal(K/F)$. Thus for any $\sigma \in G$
\[
\eta^\sigma \equiv \eta \pmod{\fP}.
\]
Hence
\[
\varepsilon=\prod_{\sigma \in G} \eta^{\sigma} \equiv \eta^{(p-1)/2} \equiv \pm 1 \pmod{\fP}
\]
as the residue field of $\fP$ is $\F_p$. Hence $\varepsilon \equiv \pm 1 \pmod{\fp}$.
\end{proof}

\begin{proof}[Proof of Theorem~\ref{T:th1}] 
Let $K=F \cdot \Q(\zeta_p)^+$.
The extension $K/F$ is Galois and cyclic. We write $G=\Gal(K/F)$ for its Galois group. 
We will show that
\begin{equation}
\label{E:ramification}
[K:F]=e(K/F)=\prod_{\upsilon} e_\upsilon=\frac{p-1}{2}.
\end{equation}

Note that $[K:\Q]$ is divisible by both $[F:\Q]$ and $[\Q(\zeta_p)^+ : \Q]=(p-1)/2$ and is a divisor of $[F:\Q] \cdot [\Q(\zeta_p)^+ : \Q]$.
By assumption (i) we see that $[K:\Q]=[F:\Q] \cdot (p-1)/2$ and so $[K:F]=(p-1)/2$. 

Since~$p$ is totally ramified in~$F$ we have $p\calO_F=\fp^{[F:\Q]}$.
Choose a prime $\fP$ of $K$ above $\fp$. As $p$ is totally ramified in $\Q(\zeta_p)^+/\Q$, and in $F/\Q$
we see that the ramification index $e(\fP/p)$ is divisible by both $[F:\Q]$ and
$[\Q(\zeta_p)^+ : \Q]=(p-1)/2$ and is a divisor of $[F:\Q] \cdot [\Q(\zeta_p)^+
: \Q]$. Thus $e(\fP/p)=[K:\Q]$ and $e(\fP/\fp)=(p-1)/2$. Since $p$ the only
prime ramified in $\Q(\zeta_p)^+/\Q$, we see that $e_\upsilon=1$ for all places
$\upsilon$ of $K$ except $\upsilon=\fP$ for which it is $(p-1)/2$.
Thus $e(K/F)=(p-1)/2=[K:F]$.


We conclude from~\eqref{E:chevalley} that $ [\calO_F^\times:\calO_F^\times \cap \Norm(K^\times)]$ divides $h_F$. 
Therefore, for all $\lambda\in \calO_F^\times$, we have
$\lambda^{h_F}\in \calO_F^\times \cap \Norm(K^\times)$, and so by 
Lemma~\ref{L:units2} applied to $\varepsilon=\lambda^{h_F}$
we have
$\lambda^{2 h_F}  \equiv  1 \pmod \fp$.
However, as the residue field of $\fp$ is $\F_p$, we also have $\lambda^{p-1} \equiv 1 \pmod{\fp}$.
By definition, $r=\gcd(2 h_F, p-1)$, and thus $\lambda^r \equiv 1 \pmod{\fp}$
for all $\lambda \in \calO_F^\times$.

Let $\lambda$, $\mu \in \calO_F^\times$ be a solution to the  unit equation~\eqref{E:unit}. 
Thus $\lambda^r \equiv (1-\lambda)^r \equiv 1 \pmod{\fp}$ and as $r$ is even (since $p \ge 3$)
we can write $\lambda^r \equiv (\lambda-1)^r \equiv 1 \pmod{\fp}$. Hence
the polynomials $X^r-1$ and $(X-1)^r-1$ have a common root in $\F_{\fp}=\F_p$,
contradicting assumption (ii).
\end{proof}

\noindent \textbf{Remark.}
Theorem~\ref{T:th1} is false without hypothesis (i). Consider the field $F=\Q(\zeta_{11})^+$ wherein the prime~$p=11$ is totally ramified 
and we have $h_F=1$, so hypothesis (ii) holds. However 
$\gcd([F:\Q],(p-1)/2) = 5$ so hypothesis (i) fails. We note the following solution to the unit equation,
$\lambda=2+\zeta_{11}+\zeta_{11}^{-1}$,
$\mu=-1-\zeta_{11}-\zeta_{11}^{-1}$. Indeed, according to
the computer algebra system \texttt{Magma} \cite{MAGMA}, 
the unit equation has $570$ solutions.

In fact, for any prime $p \ge 5$, the unit equation \eqref{E:unit}
with $F=\Q(\zeta_p)$
has the solution 
$\lambda=2+\zeta_p+\zeta_p^{-1}$, $\mu=-1-\zeta_p-\zeta_p^{-1}$.
Clearly $\lambda$, $\mu \in F$, so to see that $\lambda$, $\mu \in \calO_F^\times$, 
it is enough to show
that they are units in $\Z[\zeta_p]$. Recall that the unique
prime ideal above $p$ in $\Z[\zeta_p]$ is generated by $1-\zeta_p^j$
for any $j \not\equiv 0 \pmod{p}$. Hence
\[
\mu=-\zeta_p^{-1} (1+\zeta_p+\zeta_p^2)=-\zeta_p^{-1} \cdot \frac{(1-\zeta_p^3)}{(1-\zeta_p)}
\]
is a unit.
Further, 
\[
\lambda=(1+\zeta_p)(1+\zeta_p^{-1})=\frac{(1-\zeta_p^2)(1-\zeta_p^{-2})}{(1-\zeta_p)(1-\zeta_p^{-1})}
\] 
is also a unit. The existence of the solution $(\lambda,\mu)$ to \eqref{E:unit}
makes it difficult to establish the asymptotic Fermat's Last Theorem for
the fields $\Q(\zeta_p)^+$ with $p \ge 5$. 
 
\section{Proof of Theorem~\ref{T:th2}}
The following is a special case of \cite[Theorem 3]{FS1}.
\begin{theorem}\label{thm:FS}
Let $F$ be a totally real number field. Suppose the Eichler--Shimura conjecture over $F$ holds.
Assume that $2$ is inert in $F$ and write $\fq=2\calO_F$ for the prime ideal above $2$. Let $S=\{\fq\}$
and write $\calO_S^\times$ for the group of $S$-units in $F$. Suppose every solution $(\lambda,\mu)$
to the $S$-unit equation 
\begin{equation}\label{eqn:sunit}
\lambda+\mu=1, \qquad \lambda,~\mu \in \calO_S^\times,
\end{equation}
satisfies both of the following conditions
\[
\max\{ \lvert \ord_\fq(\lambda) \rvert,~ \lvert \ord_\fq(\mu) \rvert\} \le 4, \qquad
\ord_{\fq}(\lambda \mu)  \equiv 1 \pmod{3}.
\]
Then the asymptotic Fermat's Last Theorem holds over $F$.
\end{theorem}
For a discussion of the Eichler--Shimura conjecture see 
\cite[Section 2.4]{FS1}, but for the purpose of this paper
we note that the conjecture is known to hold for all totally real fields of odd degree.

\begin{lemma}\label{lem:sunit}
Let $F$ be a totally real number field.
Suppose that the narrow class number $h_F^+$ of $F$ is odd, and that the unit equation \eqref{E:unit}
has no solutions. 
Assume that $2$ is inert in $F$ and write $\fq=2\calO_F$ for the prime ideal above $2$. Let $S=\{\fq\}$.
For a solution $(\lambda,\mu)$ to the $S$-unit equation \eqref{eqn:sunit}, write 
\[
n_{\lambda,\mu}=\max\{ \lvert \ord_\fq(\lambda) \rvert,~ \lvert \ord_\fq(\mu) \rvert\}.
\]
Then $n_{\lambda,\mu}=1$ or $4$.
\end{lemma}
\begin{proof}
Let $(\lambda,\mu)$ be a solution to \eqref{eqn:sunit}.
If $n_{\lambda,\mu}=0$
then both $\lambda$, $\mu$ are units, and this contradicts our
assumption that the unit equation has no solutions.
Therefore $n_{\lambda,\mu} \ge 1$.

We claim that we can assume
\begin{equation}\label{E:(3.3)}
\lambda\in \calO_F \quad \text{and}\quad \ord_{\fq}(\mu)=0.
\end{equation}
Indeed, if $\ord_{\fq}(\lambda)>0$ then $\ord_{\fP}(\mu)=0$ already; if
$\ord_{\fq}(\lambda)=0$ then  $\ord_{\fq}(\mu)\geq 0$ and we simply swap
$\lambda$ and $\mu$. Suppose $\ord_{\fq}(\lambda)<0$. Then
$\ord_{\fq}(\lambda)=\ord_{\fq}(\mu)$. Moreover, we have the equality
\[
\frac{1}{\lambda}-\frac{\mu}{\lambda}=1
\]
and replacing $(\lambda,\mu)$ by $\left(\frac{1}{\lambda},-\frac{\mu}{\lambda}\right)$
gives a solution to the $S$-unit equation \eqref{eqn:sunit} satisfying~\eqref{E:(3.3)}, and the value
$n_{\lambda,\mu}$ is unchanged.

Observe that $\mu \in \calO_F^\times$. 
If $\mu \equiv 1 \pmod{\fq^2}$ then the extension $F(\sqrt{\mu})/F$
is unramified at all finite places of $F$. As $h_F^+$ is odd this
extension must be trivial, implying that $\mu$ is a square in $\calO_F^\times$.
Write $n=n_{\lambda,\mu}=\ord_{\fq}(\lambda)$ and suppose $n \ge 2$.
Thus $\mu=\delta^2$ with $\delta \in \calO_F^\times$.
Write $\lambda_1=1+\delta$ and $\lambda_2=1-\delta$.
Then $\lambda_1 \lambda_2=1-\mu=\lambda$ thus $\lambda_1$, $\lambda_2$ belong to
$\calO_S^\times$, and satisfy $\ord_\fq(\lambda_i) \ge 0$. Moreover,
\begin{equation}\label{eqn:lambdai}
\lambda_1+\lambda_2=2, \qquad \lambda_1-\lambda_2=2\delta.
\end{equation}
Since $\ord_\fq(\lambda_1)+\ord_\fq(\lambda_2)=\ord_\fq(\lambda)=n$
one of $\ord_\fq(\lambda_1)$, $\ord_\fq(\lambda_2)$ is $n-1$ and the other is $1$.
By swapping $\delta$ and $-\delta$ if necessary we may suppose
$\ord_\fq(\lambda_1)=n-1$ and $\ord_\fq(\lambda_2)=1$.
Multiplying the two equations in \eqref{eqn:lambdai} and dividing by $4 \delta$ we obtain
\[
\lambda^\prime+\mu^\prime=1, \qquad
\lambda^\prime=\frac{\lambda_1^2}{4\delta}, \quad \mu^\prime=\frac{-\lambda_2^2}{4 \delta}.
\]
Observe that we have another solution $(\lambda^\prime,\mu^\prime)$ to the $S$-unit equation
with $n_{\lambda^\prime,\mu^\prime}=2n-4$ and $\mu^\prime \in \calO_F^\times$.

If $n_{\lambda,\mu}=n=2$ then $n_{\lambda^\prime,\mu^\prime}=0$ so $(\lambda^\prime,\mu^\prime)$
is a solution to the unit equation \eqref{E:unit}, contradicting the hypothesis
that the unit equation has no solutions.
If $n_{\lambda,\mu}=n=3$ then $n_{\lambda^\prime,\mu^\prime}=2$ which
we have just established is impossible. Next suppose $n \ge 5$.
Then $n_{\lambda^\prime,\mu^\prime} =2n-4>n_{\lambda,\mu}$.
Repeating this argument yields infinitely many distinct solutions
to the $S$-unit equation, giving a contradiction. We conclude that
$n_{\lambda,\mu}=1$ or $4$ as required.
\end{proof}

Our Theorem~\ref{T:th2} follows from the following more general theorem, which
we now prove.
\begin{theorem} \label{T:th2general}
Let~$F$ be a totally real field satisfying the following conditions.
\begin{enumerate}
\item[(i)] The Eichler--Shimura conjecture over $F$ holds.
\item[(ii)] $h_F^+$ is odd.
\item[(iii)] $2$ is inert in $F$.
\item[(iv)] The unit equation has no solutions in $F$. 
\end{enumerate}
Then the asymptotic Fermat's Last Theorem over $F$ holds.
\end{theorem}
\begin{proof} 
Let $(\lambda,\mu)$ be a solution to the $S$-unit equation
\eqref{eqn:sunit}. From Lemma~\ref{lem:sunit} we have
$n_{\lambda,\mu}=1$ or $4$. As $\lambda+\mu=1$,
we conclude that 
\[
(\ord_\fq(\lambda),\ord_\fq(\mu)) \; \in \;
\left\{(-4,-4),~(-1,-1),~(1,0),~(0,1),~(4,0),~(0,4)\right\}.  
\]
Thus $\ord_\fq(\lambda\mu) \equiv 1 \pmod{3}$.
Theorem~\ref{thm:FS} now completes the proof.
\end{proof}

\bigskip\bigskip


\begin{thebibliography}{}
 
  


\bibitem{MAGMA} W.\ Bosma, J.\ Cannon et C.\ Playoust: {\em The Magma
Algebra System I: The User Language}, J.\ Symb.\ Comp.\ {\bf 24} (1997),
235--265. (see also {\tt http://magma.maths.usyd.edu.au/magma/})  







\bibitem{Evertse}
J.-H.\ Evertse, 
\emph{On equations in $S$-units
and the Thue--Mahler equation}, Invent.\ Math.\ \textbf{75}
(1984), 561--584.

\bibitem{EvertseGyory}
J.-H.\ Evertse and K.\ Gy\H{o}ry,  
\emph{Unit Equations in Diophantine Number Theory},
Cambridge Studies in Advanced Mathematics, Cambridge University Press,
2015.




\bibitem{FS1} N. Freitas and S. Siksek,
\emph{The asymptotic Fermat's Last Theorem for five-sixths of real quadratic
fields} Compositio Mathematica  \textbf{151} (2015), 1395--1415.

\bibitem{FKS} 
N. Freitas, A. Kraus and  S. Siksek,
\emph{Class field theory, Diophantine analysis and the  asymptotic Fermat's Last Theorem},
Advances in Mathematics \textbf{363} (2020), to appear.

\bibitem{FKS2} 
N. Freitas, A. Kraus and  S. Siksek,
\emph{
On Asymptotic Fermat over $\Z_p$ extensions of $\Q$ 
},
\texttt{arXiv:2003.04029}.





 


 
 


\bibitem{Lang} S.\ Lang
{\em Cyclotomic fields I and II}, 
Graduate Texts in Mathematics {\bf 121}, Springer-Verlag, New
York, 1990. 

 
 
  
  
  

 
 \bibitem{Siegel} C.\ L.\ Siegel,
\emph{\"{U}ber einige Anwendungen diophantischer
Approximationen},
Abh.\ Preuss.\ Akad.\ Wiss.\ (1929), 1--41.

\bibitem{Triantafillou}
N.\ Triantafillou, 
\emph{The unit equation has no solutions in number fields of degree
prime to $3$ where $3$ splits completely}, 
\texttt{arXiv:2003.02414}.
 





     
  

  


\end{thebibliography}
\end{document}